\documentclass[10pt,reqno]{amsart}
\usepackage{amssymb,mathrsfs,color}
\allowdisplaybreaks
\usepackage{enumerate}
\usepackage{graphicx}
\usepackage{xcolor} 
\usepackage{geometry}
\usepackage{amsfonts,amssymb,amsmath}
\usepackage{slashed}
\usepackage{tikz}
\usepackage{ mathrsfs }
\usetikzlibrary{patterns}
\usepackage[titletoc,title]{appendix}
\usepackage{wrapfig}
\usetikzlibrary{arrows,calc,decorations.pathreplacing}

\usepackage{cite}

\definecolor{green}{rgb}{0,0.8,0} 

\renewcommand{\Re}{\mathrm{Re}}
\renewcommand{\Im}{\mathrm{Im}}

\newcommand{\bfn}{{\bf n}}

\newcommand{\bfv}{{\bf v}}

\newcommand{\bfx}{{\bf x}}
\newcommand{\bfy}{{\bf y}}

\newcommand{\bbC}{\mathbb C}
\newcommand{\bbD}{\mathbb D}

\newcommand{\bbH}{\mathbb H}

\newcommand{\bbR}{\mathbb R}

\newcommand{\calA}{\mathcal A}

\newcommand{\calM}{\mathcal M}
\newcommand{\calN}{\mathcal N}

\newcommand{\Zed}{\mathfrak{Z}}
\newcommand{\zed}{\mathfrak{z}}

\newcommand{\zbar}{{\con z}}

\definecolor{deepgreen}{cmyk}{1,0,1,0.5}

\newcommand{\R}{\mathbb{R}}

\makeatletter

\newcommand{\Rmnum}[1]{\expandafter\@slowromancap\romannumeral #1@}
\makeatother

\newcommand{\con}{\overline}

\newcommand{\abs}[1]{\left\lvert{#1}\right\rvert}

\newcommand{\Aligns}[1]{\begin{align*}\begin{split} #1 \end{split}\end{align*}}

\newcommand{\Del}[1]{}

\newcommand{\pt}{&}

\numberwithin{equation}{section}

\newtheorem{theorem}{Theorem}[section]
\newtheorem{lemma}[theorem]{Lemma}
\newtheorem{proposition}[theorem]{Proposition}
\newtheorem{remark}[theorem]{Remark}

\newcommand{\pa}{\triangleright}

\renewcommand\Re{\mathrm{Re}\,}
\renewcommand\Im{\mathrm{Im}\,}

\newcommand{\ep}{\varepsilon}

\newcommand{\AV}{\mathcal{AV}}

\newcommand{\alphap}{{\alpha^\prime}}
\newcommand{\betap}{{\beta^\prime}}

\newcommand{\Zedbar}{\overline{\Zed}}

\newcommand{\Hbar}{{\overline{H}}}

\newcommand{\dt}{\partial_t}
\newcommand{\da}{\partial_\alpha}

\renewcommand{\pt}{\partial_t}

\renewcommand{\pa}{\partial_\alpha}

\newcommand{\pap}{\partial_\alphap}

\newcommand{\Gbar}{{\overline{G}}}

\newcommand{\zedbar}{{\overline{\zed}}}

\newcommand{\tOmega}{\widetilde{\Omega}}

\begin{document}
\title[The Euler-Poisson System with Constant Vorticity]{On the Motion of a Self-Gravitating Incompressible Fluid with Free Boundary and Constant Vorticity: An Appendix}

\author{Lydia Bieri}
\author{Shuang Miao}
\author{Sohrab Shahshahani}
\author{Sijue Wu}

\begin{abstract}
In a recent work \cite{BMSW1} the authors studied the dynamics of the interface separating a vacuum from 
an inviscid incompressible fluid, subject to the self-gravitational force and neglecting surface tension, in two space dimensions. The fluid is additionally assumed to be irrotational, and we proved that for data which are size $\epsilon$ perturbations of an equilibrium state, the lifespan $T$ of solutions satisfies  $T \gtrsim \epsilon^{-2}$. The key to the proof is to find a nonlinear transformation of the unknown function and a coordinate change, such that the equation for the new unknown in the new coordinate system has no quadratic nonlinear terms. For the related irrotational gravity water wave equation with constant gravity the analogous transformation was carried out by the last author in \cite{Wu3}. While our approach is inspired by the last author's work \cite{Wu3}, the self-gravity in the present problem is a new nonlinearity which needs  separate investigation. Upon completing \cite{BMSW1} we learned of  the work of Ifrim and Tataru \cite{IfrTat3} where the gravity water wave equation with constant gravity and \emph{constant vorticity} is studied and a similar estimate on the lifespan of the solution is obtained. In this short note we demonstrate that our transformations in \cite{BMSW1} can be easily modified to allow for \emph{nonzero constant vorticity}, and a similar energy method as in \cite{BMSW1} gives an estimate $T\gtrsim\epsilon^{-2}$ for the  lifespan $T$ of solutions with data which are size $\epsilon$ perturbations of the equilibrium. In particular, the effect of the constant vorticity is an extra  linear term with constant coefficient in the transformed equation, which can be further transformed away by a bounded linear transformation. 
This note serves as an appendix to the aforementioned work of the authors.
\end{abstract}

\thanks{Support of the National Science Foundation grants  DMS-1253149 for the first and second, NSF-1045119 for the third, and DMS-1361791 for the fourth authors is gratefully acknowledged. The third author was also supported by the NSF under Grant~No.0932078000 while in residence at the MSRI in Berkeley, CA during Fall 2015. The authors thank Robert Krasny for his help to improve the exposition of this note.}

\maketitle

\section{Introduction}
We study the motion of the interface separating a vacuum from an inviscid, incompressible fluid, subject to the self-gravitational force in two spatial dimensions. We assume that the fluid domain is bounded and simply connected, the surface tension is zero, and the initial vorticity of the fluid is a constant. Denoting the fluid domain by $\Omega(t),$ the fluid velocity by $\bfv,$ and the pressure by $P$, the evolution is governed by the Euler-Poisson system:

\begin{equation}\label{Euler eq}
\begin{cases}
\bfv_{t}+(\bfv\cdot\nabla)\bfv=-\nabla P-\nabla\phi\quad&\mbox{in}\quad \Omega(t),t\geq0,\\
\textrm{div}\,\bfv=0,\quad\textrm{curl}\,\bfv=2\omega_0\in\bbR\quad &\mbox{in}\quad \Omega(t), t\geq0,\\
P=0\quad &\textrm{on}\quad \partial\Omega(t),
\end{cases}
\end{equation}
where the self-gravity Newtonian potential satisfies 

\begin{equation}\label{Poisson eq}
\begin{cases}
\Delta\phi=2\pi\chi_{\Omega(t)},\\
\nabla\phi(\bfx)=\iint_{\Omega(t)}\frac{\bfx-\bfy}{\left|\bfx-\bfy\right|^{2}}d\bfy.
\end{cases}
\end{equation}
Note that applying the curl operator to the first line in equation \eqref{Euler eq} implies that the vorticity $\omega:=\mathrm{curl} \,\bfv$ satisfies $\omega_t+\bfv\cdot\nabla\omega=0,$ and therefore if $\bfv(0)$ has constant vorticity $2\omega_0$,  the same holds for $\bfv(t)$ for all times at which the latter is defined, that is $\mathrm{curl}\,\bfv=2\omega_0.$ Similarly, the incompressibility of the flow, $\mathrm{div}\,\bfv=0,$ implies that the area of the domain enclosed by the fluid does not change during the evolution. Without loss of generality we fix the normalization $|\Omega|=\pi.$ We also assume that $\omega_0^2<\pi.$ As we will show below this will be necessary to ensure the validity of the Taylor sign condition which is necessary for local well-posedness. When the vorticity is assumed to be zero a family of time-independent equilibrium solutions is given by perfect balls moving with constant velocity. In the recent preprint \cite{BMSW1} we proved that for data which are a perturbation of size $\epsilon$ of these equilibria, the lifespan $T$ of the solution satisfies the lower bound $T\gtrsim \epsilon^{-2}.$ In the irrotational case the velocity satisfies $\Delta \bfv=0$ in $\Omega(t)$ and the Euler-Poisson system \eqref{Euler eq}-\eqref{Poisson eq} can be reduced to a system on the boundary $\partial\Omega(t).$ The key to the estimate above on the lifespan of the solution is then to find a new unknown function and a change of coordinates such that the new unknown satisfies a nonlinear equation in the new coordinates with no quadratic terms in the nonlinearity. For the related irrotational gravity water wave equation with constant gravity the analogous transformation was carried out by the last author in \cite{Wu3}. While our approach is inspired by the last author's work \cite{Wu3}, the self-gravity in our problem is a new nonlinearity which requires separate investigation.  Upon completion of \cite{BMSW1} we learned of the work of  Ifrim and Tataru \cite{IfrTat3}, where the gravity water wave equation with constant gravity and \emph{constant vorticity} is studied and a lifespan estimate  $T\gtrsim\epsilon^{-2}$ is proved for data which are perturbations of size $\epsilon$ of the equilibrium.  In this short note we demonstrate that our transformations in \cite{BMSW1} can be easily modified to allow for nonzero constant vorticity, and a similar energy method as in \cite{BMSW1} gives  an estimate $T\gtrsim\epsilon^{-2}$ for the  lifespan $T$ of solutions with data which are size $\epsilon$ perturbations of the equilibrium.

In the remainder of this note we adopt the notation introduced in \cite{BMSW1} and to avoid repetition refer the reader to \cite{BMSW1} for the definitions of the symbols we use. Note that the vector field $\bfv_0:=\omega_0(y, -x)$ satisfies $\mathrm{curl}\,\bfv_0=2\omega_0$ and $\mathrm{div}\,\bfv_0=0,$ and therefore $\frak v:=\bfv-\bfv_0$ is curl and divergence free. Using complex variable notation it follows that $z_t+i\omega_0z$ is the boundary value of an anti-holomorphic function in $\Omega,$ or in other words $(I-\Hbar)(z_t+i\omega_0z)=0,$ where $H$ denotes the Hilbert transform for $\partial\Omega.$ With this observation, and with the notation 

\Aligns{
a:=-\frac{1}{|z_\alpha|}\frac{\partial P}{\partial {\bold n}},
}
the system \eqref{Euler eq} reduces to the following system on the boundary $\partial\Omega:$

\begin{equation}\label{z cv eq}
\begin{cases}
z_{tt}+iaz_\alpha=-\frac{\pi}{2}(I-\Hbar)z\\
H(\zbar_t-i\omega_0\zbar)=\zbar_t-i\omega_0\zbar
\end{cases}.
\end{equation}
We refer the reader to \cite{BMSW1} for the derivation of \eqref{z cv eq}. Note that $z(t,\alpha)=e^{-i\omega_0t+i\alpha}$ is a solution to \eqref{z cv eq} with $a=\pi-\omega_0^2.$ It follows that for the Taylor sign condition $\frac{\partial P}{\partial \bfn}<0$ to hold we need to impose the condition $\omega_0^2<\pi.$ The following theorem is the main result of this note.

\begin{theorem}\label{thm: main cv}
Let $\Omega_0$ be a bounded simply-connected domain in $\bbC$ with smooth boundary $\partial\Omega_0$ satisfying $\abs{\Omega_0}=\pi,$ and denote the associated Hilbert transform by $H_0.$ Suppose $z_0(\alpha)=e^{i\alpha}+\epsilon f(\alpha)$ is a parametrization of $\partial\Omega_0$ and $z_1(\alpha)=v_0-i\omega_0e^{i\alpha}+ \epsilon g(\alpha)$ where $f$ and $g$ are smooth, $H_0(\zbar_1-i\omega_0\zbar_0)=\zbar_1-i\omega_0\zbar_0,$ $v_0\in\bbC$ is a constant, and $\omega_0^2<\pi$. Then there is $T>0$ and a unique classical solution $z(t,\alpha)$ of \eqref{z cv eq} on $[0,T)$ satisfying $(z(0,\alpha),z_t(0,\alpha))=(z_0(\alpha),z_1(\alpha)).$  Moreover if $\epsilon>0$ is sufficiently small the solution can be extended at least to $T^*=c\epsilon^{-2}$ where $c$ is a constant independent of $\epsilon.$
\end{theorem}

\begin{remark}
As in \cite{BMSW1} the normalization $|\Omega|=\pi$ is for notational convenience and can be removed, and the constant $v_0$ is to account for the equilibrium solutions described by balls moving at constant velocity. 
\end{remark}

\section{Proof of Theorem~\ref{thm: main cv}}
In this section we present the proof of Theorem~\ref{thm: main cv}. The proof follows closely the proof of Theorem 1.1 in \cite{BMSW1} and here we only provide the extra computations needed for the argument in \cite{BMSW1} to also prove Theorem~\ref{thm: main cv} above.

We introduce the notation $\zed(t,\alpha)=e^{i\omega_0t}z(t,\alpha)$ and note that $H\zedbar_t=\zedbar_t.$ Note that since the factor $e^{i\omega_0t}$ is independent of $\alpha$ one can replace $z$ by $\zed$ in the definition of the Hilbert transform, and in particular the conclusions of Lemma 3.7 in \cite{BMSW1} remain valid with $z$ replaced by $\zed.$ The system \eqref{z cv eq} is written in terms of $\zed$ as

\begin{align}\label{zed cv eq}
\begin{cases}
\zed_{tt}+ia\zed_\alpha=-\frac{\pi}{2}(I-\Hbar)\zed+2i\omega_0\zed_{t}+\omega_{0}^{2}\zed\\
H\zedbar_t=\zedbar_t
\end{cases}.
\end{align}
We first show the validity of the Taylor sign condition provided $\omega_0^2<\pi,$ from which local well-posedness follows as in \cite{BMSW1}. Let $\tOmega(t)$ be the domain with $\partial\tOmega(t)$ parametrized by $\zed(t,\cdot)$. We introduce the Riemann mapping $\Phi(t,\cdot): \tOmega(t)\rightarrow\bbD$, the function $h: \bbR\rightarrow\bbR$ such that

\begin{align}\label{function h}
e^{ih(t,\alpha)}=\Phi(t,\zed(t,\alpha)),
\end{align}
and the new unknowns in the Riemann mapping coordinates:

\begin{align}\label{Riemann mapping unknowns}
\begin{split}
&\Zed(t,\alphap)=\zed(t,h^{-1}(t,\alpha)),\quad \Zed_{t}(t,\alphap)=\zed_{t}(t,h^{-1}(t,\alpha)),\\
&\Zed_{tt}(t,\alphap)=\zed_{tt}(t,h^{-1}(t,\alpha)),\quad \Zed_{ttt}(t,\alphap)=\zed_{ttt}(t,h^{-1}(t,\alpha)).
\end{split}
\end{align}
The new unknowns satisfy the following system on the unit circle

\begin{align}\label{Zed eq}
\begin{cases}
&\Zed_{tt}+i\calA\Zed_{,\alphap}=-(\pi-\omega_{0}^{2})\Zed+\Gbar,\\
&\bbH\Zedbar_{t}=\Zedbar_{t}.
\end{cases}
\end{align}
where $\calA\circ h=ah_\alpha$, $\bbH$ is the Hilbert transform associated to the unit circle, and $G$ is given by

\begin{align*}
G=\frac{\pi}{2}\left((I+H)\zedbar\right)\circ h^{-1}-2\omega_{0}i\zedbar_{t}\circ h^{-1}.
\end{align*}
Note that $\Gbar$ is the boundary value of an anti-holomorphic function in the unit disc, and therefore a similar argument as the proof of Proposition 7.1 in \cite{BMSW1} gives the following result.

\begin{proposition}
$\calA_{1}:=\calA\left|\Zed_{,\alphap}\right|^{2}$ is positive and is given by

\begin{align*}
\calA_{1}:=&\frac{1}{8\pi}\int_{0}^{2\pi}\left|\Zed_{t}(t,\betap)-\Zed_{t}(t,\alphap)\right|^{2}\csc^{2}\left(\frac{\betap-\alphap}{2}\right)d\betap\\+&\frac{\pi-\omega_{0}^{2}}{8\pi}\int_{0}^{2\pi}\left|\Zed(t,\betap)-\Zed(t,\alphap)\right|^{2}\csc^{2}\left(\frac{\betap-\alphap}{2}\right)d\betap>0.
\end{align*}
\end{proposition}

We next turn to the proof of long-time existence for small initial data. Note that $\zed_t$ is now the small quantity corresponding to $z_t$ in the irrotational case. The key point for the proof of Theorem~\ref{thm: main cv} is that $\delta:=(I-H)\ep,$ where $\ep:=|z|^2-1=|\zed|^2-1,$ still satisfies an equation without quadratic nonlinear terms. The following proposition is the analogue of Proposition 3.15 in \cite{BMSW1}.

\begin{proposition}\label{prop: delta cv eq}
Let $E$ be as in Lemma 3.13. in \cite{BMSW1}. Then the quantities $\delta$ and $\delta_t=\pt \delta$ satisfy

\begin{align}\label{delta cv eq 1}
\begin{split}
 (\pt^2+ia\pa-\pi-2\omega_{0}i\pt)\delta=\widetilde\calN_1:=&\frac{\pi}{2}(I-H)E(\ep)+\frac{\pi}{2}[E(\zed),H]\frac{\ep_\alpha}{\zed_\alpha}-2[\zed_t,H\frac{1}{\zed_\alpha}+\Hbar\frac{1}{\zedbar_\alpha}]\partial_\alpha(\zed_t\zedbar)\\
&-\frac{1}{\pi i}\int_0^{2\pi}\left(\frac{\zed_t(\beta)-\zed_t(\alpha)}{\zed(\beta)-\zed(\alpha)}\right)^2\ep_\beta(\beta)d\beta,
\end{split}
\end{align}
 and
 
\begin{align}\label{deltat cv eq 1}
\begin{split}
(\pt^2+ia\pa-\pi-2\omega_{0}i\pt)\delta_t=\widetilde\calN_2:=&-ia_t\partial_\alpha\delta+\frac{\pi}{2}\left((I-H)\dt E(\ep)-[\zed_t,H]\frac{\da E(\ep)}{\zed_\alpha}\right)+\frac{\pi}{2}[\dt E(\zed),H]\frac{\ep_\alpha}{\zed_\alpha}\\
&+\frac{\pi}{2}[E(\zed),H]\partial_t\left(\frac{\ep_\alpha}{\zed_\alpha}\right)+\frac{\pi}{2}E(\zed)[\zed_t,H]\frac{\partial_\alpha\left(\frac{\ep_\alpha}{\zed_\alpha}\right)}{\zed_\alpha}-\frac{\pi}{2}[\zed_t,H]\frac{\partial_\alpha\left(E(\zed)\frac{\ep_\alpha}{\zed_\alpha}\right)}{\zed_\alpha}\\
&+\frac{2}{\pi i}\partial_t\int_0^{2\pi}\frac{(\zed_t(\alpha)-\zed_t(\beta))\partial_\beta(\zed_t(\beta)\zedbar(\beta))\zedbar(\alpha)\zedbar(\beta)\left(\frac{\ep(\alpha)}{\zedbar(\alpha)}-\frac{\ep(\beta)}{\zedbar(\beta)}\right)}{|\zed(\beta)-\zed(\alpha)|^2}d\beta\\
&-\frac{1}{\pi i}\partial_t\int_0^{2\pi}\left(\frac{\zed_t(\beta)-\zed_t(\alpha)}{\zed(\beta)-\zed(\alpha)}\right)^2\ep_\beta(\beta)d\beta,  
\end{split}
\end{align}

\end{proposition}
\begin{remark}
We note that the vorticity $2\omega_0$ gives the extra linear terms $-2\omega_0 i\partial_t\delta$ in \eqref{delta cv eq 1} and $-2\omega_0 i\partial_t\delta_t$ in \eqref{deltat cv eq 1}.

\end{remark}

\begin{proof}
Using \eqref{zed cv eq}, we write the equation satisfied by $\ep$ as

\begin{align*}
\begin{split}
 (\pt^2+ia\pa)\ep=\frac{\pi}{2}\left(\zed(I-H)\zedbar-\zedbar(I-\Hbar)\zed\right)+2\omega_{0}i\pt\ep+2(\zedbar_{t}\zed)_{t}, 
\end{split}
\end{align*}
which implies

\begin{align*}
\begin{split}
  (I-H)(\pt^2+ia\pa)\ep=\frac{\pi}{2}(I-H)\left(\zed(I-H)\zedbar-\zedbar(I-\Hbar)\zed\right)+2\omega_{0}i\pt\delta
 +2[\zed_{t},H]\frac{(\zedbar_{t}\zed)_{\alpha}}{\zed_{\alpha}}.
\end{split}
\end{align*}
Arguing as in  Proposition 3.15 of \cite{BMSW1} and using Lemma 3.7 of \cite{BMSW1} we get

\begin{align*}
\begin{split}
(\pt^2+ia\pa-2\omega_{0}i\pt)\delta  =&\frac{\pi}{2}(I-H)\left(\zed(I-H)\zedbar-\zedbar(I-\Hbar)\zed\right)+\frac{\pi}{2}[(I-\Hbar)\zed,H]\frac{\ep_\alpha}{\zed_\alpha}\\
&-2[\zed_t,H\frac{1}{\zed_\alpha}+\Hbar\frac{1}{\zedbar_\alpha}]\partial_\alpha(\zed_t\zedbar)-\frac{1}{\pi i}\int_0^{2\pi}\left(\frac{\zed_t(\beta)-\zed_t(\alpha)}{\zed(\beta)-\zed(\alpha)}\right)^2\ep_\beta(\beta)d\beta.
\end{split}
\end{align*}
Exactly the same computation as in the proof of Proposition 3.15 in \cite{BMSW1} now shows that

\begin{align*}
\begin{split}
\frac{\pi}{2}(I-H)\left(\zed(I-H)\zedbar-\zedbar(I-\Hbar)\zed\right)+\frac{\pi}{2}[(I-\Hbar)\zed,H]\frac{\ep_\alpha}{\zed_\alpha}=\pi\delta+\frac{\pi}{2}(I-H)E(\ep)+\frac{\pi}{2}[E(\zed),H]\frac{\ep_\alpha}{\zed_\alpha},
\end{split}
\end{align*}
where $E$ is as in Lemma 3.13 in \cite{BMSW1}. Combined with the previous identity this gives equation \eqref{delta cv eq 1}, and \eqref{deltat cv eq 1} follows from differentiating \eqref{delta cv eq 1} and using Lemma 3.7 in \cite{BMSW1}.
\end{proof}

To show that $a_t$ does not contribute quadratic terms to the nonlinearity in the equation for $\delta_t$ we record the following analogue of Lemma 3.16 in \cite{BMSW1}. 

\begin{lemma}\label{lem: at cv}
Let $K^*$ denote the formal adjoint of $K:=\Re H=\frac{1}{2}(H+\Hbar)$. Then
\Aligns{
(I+K^*)(a_t|\zed_\alpha|)=\Re\Bigg[\frac{-i\zed_\alpha}{|\zed_\alpha|}\Big\{&2[\zed_t,H]\frac{\zedbar_{tt\alpha}}{\zed_\alpha}+2[\zed_{tt},H]\frac{\zedbar_{t\alpha}}{\zed_\alpha}-[e^{it}g^a,H]\frac{\zedbar_{t\alpha}}{\zed_\alpha}+2\omega_{0}i[\zed_t,H]\frac{\zedbar_{t\alpha}}{\zed_\alpha}\\
&+\frac{1}{\pi i}\int_0^{2\pi}\left(\frac{\zed_t(\beta)-\zed_t(\alpha)}{\zed(\beta)-\zed(\alpha)}\right)^2\zedbar_{t\beta}(\beta)d\beta+\frac{\pi}{2}([\zed_t,H]\frac{\pa g^{h}}{\zed_\alpha})\Big\}\Bigg].
}
\end{lemma}

\begin{proof}
The proof is the same as that of Lemma 3.16 in \cite{BMSW1}. The only modification is that differentiating \eqref{zed cv eq} in the time variable we get the following equation for $\zedbar_t:$

\begin{align*}
\begin{split}
 \zedbar_{ttt}-ia\zedbar_{t\alpha}=ia_t\zedbar_\alpha+\frac{\pi}{2}[\zed_t,H]\frac{\zedbar_\alpha}{\zed_\alpha}+\zedbar_t-2\omega_{0}i\zedbar_{tt}. 
\end{split}
\end{align*}
The only new term compared to the equation for $\zbar_t$ in Lemma 3.16 in \cite{BMSW1} is the last term on the right hand side. For this note that if $F(t,\zed)$ is the holomorphic function with boundary value $\zedbar_t$ then

\begin{align*}
\begin{split}
\zedbar_{tt}=F_t+\frac{\zedbar_{t\alpha}\zed_t}{\zed_\alpha}=F_t+\frac{\zedbar_{t\alpha}\zed_t}{\zed_\alpha}. 
\end{split}
\end{align*}
Since $\zedbar_t$ and $\zed$ have the same holomorphicity properties as $\zbar_t$ and $z$ in the irrotational case, the rest of the proof is exactly the same as that of Lemma 3.16 in \cite{BMSW1}.
\end{proof}

We now define the change of coordinates $k$ similarly to Remark 3.21 in \cite{BMSW1} and show that in the new coordinate $\alphap=k(t,\alpha)$ the equations for $\delta$ and $\delta_t$ have no quadratic nonlinearities. The precise identities are given in the following proposition.

\begin{proposition}\label{prop: k cv}
Suppose $z(t,\cdot)$ is a simple closed curve containing the origin in its simply connected interior for each $t\in I,$ where $I$ is some time interval, and let $k$ be as defined in Remark 3.21 in \cite{BMSW1}, but with $z$ replaced by $\zed,$ that is $(I-H)(\log(\zedbar e^{ik}))=0$. Then

\begin{align*}
\begin{split}
&(I-H)k_t= -i(I-H)\frac{\zedbar_{t}\ep}{\zedbar }-i[\zed_{t},H]\frac{\left(\log(\zedbar e^{ik})\right)_{\alpha}}{\zed_{\alpha}},\\
&(I-H)k_{tt}=-i(I-H)\frac{\zedbar_{tt}\ep+\zedbar_{t}\ep_{t}}{\zedbar}+i(I-H)\frac{\zedbar_{t}^{2}\ep}{\zedbar^{2}}-i[\zed_{t},H]\frac{(\log(\zedbar e^{ik}))_{t\alpha}+ik_{t\alpha}}{\zed_{\alpha}}+i[\zed_{t},H]\frac{1}{\zed_{\alpha}}\pa\left(\frac{\zedbar_{t}\ep}{\zedbar}\right)\\
&\qquad\qquad\qquad\quad-i[\zed_{tt},H]\frac{(\log(\zedbar e^{ik}))_{\alpha}}{\zed_{\alpha}}-\frac{1}{\pi}\int_{0}^{2\pi}\left(\frac{\zed_{t}(\beta)-\zed_{t}(\alpha)}{\zed(\beta)-\zed(\alpha)}\right)^{2}(\log(\zedbar e^{ik}))_{\beta}d\beta,\\
&(I-H)(ak_\alpha)=[\zed_{t},H]\frac{(\zedbar_{t}\zed)_{\alpha}}{\zed_{\alpha}}-[\zed_t,H]\zedbar_t\\
&\qquad\qquad\qquad\quad-(I-H)\frac{(\zedbar_{tt}+2\omega_{0}i\zedbar_t)\ep}{\zedbar}+(I-H)\frac{e^{-it}g^{h}\ep}{\zedbar}+[\zed_{tt}-2\omega_{0}i\zed_t-e^{it}g^{a},H]\frac{\left(\log(\zedbar e^{ik})\right)_{\alpha}}{\zed_{\alpha}}.
\end{split}
\end{align*}
Moreover if $F$ is a holomorphic function with boundary value $\zedbar e^{ik}$ and satisfies $F(t,0)\in\R_{+}$ for all $t\in I,$ then with the notation $\AV(f):=\frac{1}{2}[z,H]\frac{f}{z}$,

\Aligns{
&\AV (\ep)=0,\\
&\AV(a k_\alpha)=-(\pi-\omega_{0}^{2})+\frac{\omega_{0}}{\pi}\int_0^{2\pi}\frac{\zedbar_t\ep\zed_\beta}{|\zed|^2}d\beta+\frac{1}{2\pi i}\int_0^{2\pi}\zedbar_t\zed_{t\beta}d\beta+\frac{1}{2\pi i}\int_0^{2\pi}\frac{(\zedbar_{tt}-e^{-it}g^h)\ep \zed_\beta}{|\zed|^2}d\beta,\\
&\qquad\qquad\quad-\frac{1}{2\pi i}\int_0^{2\pi}\left(\frac{\zed_{tt}-2i\zed_t-e^{it}g^a}{\zed}\right)\partial_\beta\log Fd\beta,\\
&\Re \AV (k_t)=\frac{\Re}{2\pi }\int_0^{2\pi}\frac{\zedbar_t\ep}{|\zed|^2}\zed_\beta d\beta-\frac{\Re}{2\pi}\int_0^{2\pi}\log F \left(\frac{\zed\zed_{t\beta}-\zed_t\zed_\beta}{\zed^2}\right)d\beta,\\
&\Re\AV(k_{tt})=\frac{\Im}{2\pi}\int_0^{2\pi}\left(\frac{\zed_{t\beta}\zed-\zed_t\zed_\beta}{\zed^2}\right)k_td\beta+\frac{\Re}{2\pi}\pt\int_0^{2\pi}\frac{\zedbar_t\ep}{|\zed|^2}\zed_\beta d\beta+\frac{\Re}{2\pi}\pt \int_0^{2\pi}(\log (\zedbar e^{ik}))_\beta \frac{\zed_t}{\zed}d\beta.
} 
\end{proposition}

\begin{proof}
The proof is the the same as those of Propositions 3.18, 3.20, and 5.14 in \cite{BMSW1}. Indeed for the identities involving $k_{t}$ and $k_{tt}$ it suffices to note that $\zed_t$ and $\zed$ have the same holomorphicity as $z_{t}$ and $z$ as in \cite{BMSW1} and the derivation does not rely on the first equations in \eqref{z cv eq} and \eqref{zed cv eq}. The second identity follows from the same argument as in Proposition 3.18 of \cite{BMSW1}. The only difference is that instead of $-\pi z+g^{a}$, the right hand side of the first equation in \eqref{zed cv eq} can be written as $-(\pi-\omega^{2}_{0})\zed+e^{it}g^{a}+2i\zed_{t}$. The computation of the averages follows from similar modifications of the proof of Proposition 3.20 in \cite{BMSW1}. 
\end{proof}

Note that the computations for $\AV(ak_\alpha)$ and for the static solution in the introduction show that $a$ is close to $\pi-\omega^{2}_{0},$ whereas the ``negative Klein-Gordon" term in the equation for $\delta$ is still $-\pi\delta.$ This can be simply rectified by introducing the new unknown

\begin{align*}
\begin{split}
\tilde\delta:=e^{-\omega_{0}it}\delta.  
\end{split}
\end{align*}
With this definition $\tilde\delta$ satisfies

\begin{align}\label{tdelta cv eq}
\begin{split}
&(\pt^2+ia\pa-(\pi-\omega^{2}_{0}))\tilde\delta=\widetilde\calM_1:=
e^{-\omega_{0}it}\widetilde \calN_1,\\
&(\pt^2+ia\pa-(\pi-\omega_{0}^{2}))\tilde\delta_t=\widetilde\calM_2:=
e^{-\omega_{0}it}(\widetilde \calN_2+i\widetilde\calN_1).
\end{split}
\end{align}
With the same notation as in \cite{BMSW1} and with $\widetilde N_j =\widetilde\calM_j\circ k^{-1},~j=1,2,$ we rewrite the equations for $\chi:=\tilde\delta\circ k^{-1}$ and $v=(\pt\tilde\delta)\circ k^{-1}$  as

\begin{align}\label{chi v cv eq}
\begin{split}
&(\pt+b\pap)^2\chi+iA\pap\chi-(\pi-\omega_{0}^{2})\chi=\widetilde{N}_1,\\
&(\pt+b\pap)^2v+iA\pap v- (\pi-\omega_{0}^{2}) v=\widetilde N_2.
\end{split}
\end{align}
We can now prove Theorem~\ref{thm: main cv}.

\begin{proof}[Proof of Theorem~\ref{thm: main cv}]
Since equation \eqref{chi v cv eq} has the same form as the equation studied in \cite{BMSW1} the energy estimates are exactly the same. In view of Propositions~\ref{prop: delta cv eq} and~\ref{prop: k cv} and Lemma~\ref{lem: at cv} the right hand sides of the equations in~\eqref{chi v cv eq} contain no quadratic terms. Similarly, the identity for $k_{tt}$ in Proposition~\ref{prop: k cv} shows that the contributions of $(\pt+b\pap)b$ which arise in the higher energy estimates as in Proposition 5.15 of \cite{BMSW1} do not contain quadratic terms. Therefore the only remaining step in the proof is to verify that $\tilde v:=(I-H)v$ also satisfies an equation with no quadratic nonlinearities, analogous to the equation derived in Proposition 5.11 in \cite{BMSW1}. The computation here is similar and we only present the necessary modifications. We use the same proof as in Proposition 5.11 in \cite{BMSW1} replacing $z$ by $\zed$ throughout. In the first step in the commutator $[\pt^2+ia\pa-(\pi-\omega_{0}^{2}),H]\tilde\delta_t$ we get the extra term $2\omega_0i[\zed_t,H]\frac{\tilde\delta_{t\alpha}}{\zed_\alpha}$ compared to \cite{BMSW1}. This can be written as

\begin{align*}
\begin{split}
[\zed_t,H]\frac{\tilde\delta_{t\alpha}}{\zed_\alpha}=&-\omega_{0}i[\zed_t,H]\frac{\tilde\delta_\alpha}{\zed_\alpha}+[\zed_t,H]\frac{e^{-\omega_{0}it}\pa(I-H)\ep_t}{\zed_\alpha} -[\zed_t,H]\frac{e^{-\omega_{0}it}[\zed_t,H]\frac{\ep_\alpha}{\zed_\alpha}}{\zed_\alpha}.
\end{split}
\end{align*}
The last term is already cubic. To see that $[\zed_t,H]\frac{\tilde\delta_\alpha}{\zed_\alpha}$ is also cubic note that
 
\begin{align*}
\begin{split}
e^{\omega_{0}it}\tilde\delta=(I-H)\ep=(I+\Hbar)\ep-(H+\Hbar)\ep=(I+\Hbar)\ep-\zed[\ep,H]\frac{\ep_\alpha}{\zed_\alpha}+E(\ep).
\end{split}
\end{align*}
The contribution of $-\zed[\ep,H]\frac{\ep_\alpha}{\zed_\alpha}+E(\ep)$ to $[\zed_t,H]\frac{\tilde\delta_\alpha}{\zed_\alpha}$ is clearly cubic and for $(I+\Hbar)\ep$ note that

\begin{align*}
\begin{split}
[\zed_t,H]\frac{e^{-\omega_{0}it}\pa(I+\Hbar)\ep}{\zed_\alpha}=[\zed_t,H\frac{1}{\zed_\alpha}+\Hbar\frac{1}{\zedbar_\alpha}]e^{-\omega_{0}it}\pa(I+\Hbar)\ep,
\end{split}
\end{align*}
which is cubic. The contribution of $[\zed_t,H]\frac{e^{-\omega_{0}it}\pa(I-H)\ep_t}{\zed_\alpha}$ is shown to be cubic in a similar way. The only other computation which is different from the proof of Proposition 5.11 in \cite{BMSW1} is that of the term $II:=-2[\zed_t,H]\frac{\pa(ia\pa\tilde\delta)}{\zed_\alpha}$ on page 46 of \cite{BMSW1}, where we use equation \eqref{zed cv eq} for $\zed$ instead of the equation for $z$ in \cite{BMSW1}. Here the extra terms we get are

\begin{align*}
\begin{split}
-2\omega_{0}i[\zed_t,H]\frac{\pa}{\zed_\alpha}\left(\frac{\zed_t\tilde\delta_\alpha}{\zed_\alpha}\right)-\omega_{0}^{2}[\zed_t,H]\frac{\pa}{\zed_\alpha}\left(\frac{\zed\tilde\delta_\alpha}{\zed_\alpha}\right). 
\end{split}
\end{align*}
The first term is already cubic and the second term is identical, up to a multiplicative constant, to one of the terms already computed in the calculation of $II$ in \cite{BMSW1}. This shows that the equation for $\tilde v$ contains no quadratic nonlinearities and this completes the proof of Theorem~\ref{thm: main cv}.
\end{proof}
\bibliographystyle{plain}
\bibliography{twobodybib}

 \bigskip

\centerline{\scshape Lydia Bieri, Shuang Miao, Sohrab Shahshahani, Sijue Wu}
\medskip
{\footnotesize
 \centerline{Department of Mathematics, The University of Michigan}
\centerline{2074 East Hall, 530 Church Street
Ann Arbor, MI  48109-1043, U.S.A.}
\centerline{\email{lbieri@umich.edu, shmiao@umich.edu, shahshah@umich.edu, sijue@umich.edu}}

\end{document}